\documentclass[a4,12pt]{article}       

\usepackage[utf8]{inputenc}
\usepackage[T1]{fontenc}

\usepackage[english]{babel}

\usepackage{amsmath}
\usepackage{amsthm}
\usepackage{amssymb}

\usepackage{enumitem}
\usepackage{graphicx}
\usepackage{mathtools}
\usepackage{hyperref}

\usepackage[all]{xy}

\DeclareMathOperator{\coker}{coker}
\DeclareMathOperator{\im}{im}
\DeclareMathOperator{\Cone}{Cone}
\DeclareMathOperator{\Ker}{Ker}
\DeclareMathOperator{\End}{End}
\DeclareMathOperator{\Hom}{Hom}
\DeclareMathOperator{\Ext}{Ext}
\DeclareMathOperator{\Tor}{Tor}

\newtheorem{theorem}{Theorem}[section]
\newtheorem{lemma}[theorem]{Lemma}
\newtheorem{proposition}[theorem]{Proposition}

\newtheorem{definition}[theorem]{Definition}
\newtheorem{remark}[theorem]{Remark}
\newtheorem{example}[theorem]{Example}

\begin{document}

\title{Tilting modules  and tilting torsion pairs
\\Filtrations induced by tilting modules}


\author{Francesco Mattiello  \and
	Sergio Pavon         \and
        Alberto Tonolo 
}


%

\maketitle

\begin{abstract}
	Tilting modules, generalising the notion of progenerator, furnish equivalences between pieces of
	module categories. This paper is dedicated to study how much these pieces say about the
	whole category. We will survey the existing results in the literature, introducing also some
	new insights.
\end{abstract}

\section*{Introduction}
\label{intro}
In 1958 Morita characterised equivalences between the entire categories of left (or right) modules
over two rings. Let $A$ be an arbitrary associative ring with $1\not=0$. A left $A$-module ${}_AP$
is a \emph{progenerator} if it is projective, finitely generated and generates the category $A$-Mod
of left $A$-modules. Set $B:=\textrm{End}({}_AP)$, the covariant functor $\Hom_A(P,?)$ gives
an equivalence between $A$-Mod and $B$-Mod; moreover any equivalence between modules categories is
of this type.

The notion of tilting module has been axiomatised in 1979 by Brenner and Butler [BB], generalising
that of progenerator for modules of projective dimension 1. The various form of generalisations to
higher projective dimensions considered until today continue to follow their  approach. 

A tilting module $T$ of projective dimension $n$ naturally gives rise to $n+1$ corresponding classes of
modules in $A$-Mod and $B$-Mod, the Miyashita classes, with $n+1$ equivalences between them. These classes are
\[\text{KE}_e(T)=\{M\in A\text{-Mod}:\Ext^i_A(T,M)=0\ \forall i\not=e\}{\phantom{, \quad e=0,1,...,n.}}\]
\[
\text{KT}^e(T)=\{N\in B\text{-Mod}:\Tor_i^B(T,M)=0\ \forall i\not=e\},\quad e=0,1,...,n\]
and the $n+1$ equivalences are
\[\xymatrix@C=5pc{KE_e(T)\ar@<.2pc>[r]^-{\Ext^e_A(T,?)}&KT^e(T)\ar@<.2pc>[l]^-{\Tor_e^B(T,?)}}, \quad e=0,1,...,n.\]

In the $n=0$ case (progenerator), there is only one class on each side, and so every module is subject to the
equivalence of categories (that of Morita); for $n=1$, on each side the two Miyashita classes form torsion pairs, so every module in both $A$-Mod
and $B$-Mod can be decomposed in terms of modules in the Miyashita classes: precisely every module admits a composition series of length 2 with composition factors in the Miyashita classes.

For $n>1$, the Miyashita classes fail to decompose every module; the way to
recover a similar decomposition is the subject of this paper.

In Section \ref{sec:1}, we define classical $n$-tilting modules and Miyashita classes; we show
that they give a torsion pair for $n=1$, and hence they can be used to decompose every module;
we give an example showing that a similar decomposition does not exist for $n>1$, and characterise
those modules which can be decomposed.

In Section \ref{sec:2}, we present some previous attempts to recover the decomposition for $n>1$ as
well, by extending the Miyashita classes, due to Jensen, Madsen, Su \cite{JensenMadsenSu13} and to
Lo \cite{Lo15}. A useful tool in our analysis will be a characterisation of modules in
$\cap_{i>e}\Ker\Ext_A^i(T,?)$, $0\leq~e\leq n$,  (see Lemma~\ref{lemma:description}) obtained
generalising the characterisation of modules in $\cap_{i>0}\Ker\Ext_A^i(T,?)$ given by Bazzoni in
\cite[Lemma 3.2]{Bazzoni04}. These extensions deform in an irreversible way the Miyashita classes,
weakening their role.

In Section \ref{sec:3}, we recall some introductory notions about the derived category of an abelian
category, and about $t$-structures.

In Section \ref{sec:4}, we consider a generalisation of the notion of classical $n$-tilting modules,
to define non classical ones. In this setting,
we define the $t$-structure associated to a tilting module; we then study its interaction with the natural
$t$-structure of the derived category.

In Section \ref{sec:5}, we exploit the results of Section \ref{sec:4} to construct in the derived
category the \emph{t-tree} of a module with respect to a tilting module.
This procedure, discovered in the classical
tilting case by Fiorot, the first and the third author in \cite{FiorotMattielloTonolo16},
solves satisfactorily the decomposition problem for $n>1$: the classes used for the decomposition
intersect the module category exactly in the Miyashita classes. As a result of the work of the
previous section, we prove that this construction can be reproduced also in the non classical case.

Throughout the paper, the concrete case considered in Example \ref{ex:2tilt} introduced in Section \ref{sec:1} will be used to
illustrate the various attempts to solve the decomposition problem (see Examples~\ref{ex:JMS}, \ref{ex:ttree}).

\section{Classical $n$-tilting modules} \label{sec:1}
In 1986, Miyashita \cite{Miyashita86} and Cline, Parshall and Scott \cite{ClineParshallScott86} gave similar definitions
of a \emph{tilting module of projective dimension $n$}.

\begin{definition}[Miyashita \cite{Miyashita86}]\label{definition:classical}
	A left $A$-module $T$ is a classical $n$-tilting module, for some integer $n\geq 0$, if:
	\begin{enumerate}
		\item[$p_n$)] $T$ has a finitely generated projective resolution of length $n$, i.e. a projective resolution
			$$\xymatrix{0\ar[r]&P_n\ar[r]&\cdots\ar[r]&P_0\ar[r]&T\ar[r]&0}$$
			with the $P_i$ finitely generated;
		\item[$e_n$)] $T$ is rigid, i.e. $\Ext_A^i(T,T)=0$ for every $0<i\leq n$;
		\item[$g_n$)] the ring $A$ admits a coresolution of length $n$
			$$\xymatrix{0\ar[r]&A\ar[r]&T_0\ar[r]&\cdots\ar[r]&T_n\ar[r]&0}$$
			with the $T_i$ finitely generated direct summands of arbitrary coproducts of copies of $T$.
	\end{enumerate}
\end{definition}

In the case when $n=0$, $p_0$) says that the module is a finitely generated projective, $e_0$) is
empty and $g_0$) says that it is a generator: this is then the definition of a progenerator
module. As such, a classical $0$-tilting module $T$ induces a Morita equivalence of categories of modules, as
follows. Let $B=\End_A(T)$ be its ring of endomorphisms, which acts on the right on $T$, and consider
the category $B$-Mod of left $B$-modules. There are functors
\begin{align*}
	\Hom_A(T,?)&:\quad A\text{-Mod}\to B\text{-Mod}\\
	T\otimes_B?&:\quad B\text{-Mod}\to A\text{-Mod}
\end{align*}
which are category equivalences, with the unit and counit morphisms being those of the adjunction.
This is the motivating example for the definition of tilting modules, along with the next case.

In the case when $n=1$, we find what was originally (see Brenner and Butler, \cite{BrennerButler80})
defined as a \emph{tilting module}; we will give a brief and incomplete overview of what is known
about them.

Let $T$ be a classical 1-tilting left $A$-module, and let as before $B=\End_A(T)$ be its ring of
endomorphisms. In this case $T$ does not induce an equivalence of $A$-Mod and $B$-Mod anymore;
however, a little less can be proved, as follows.

Define the following pairs of full subcategories of $A$-Mod and $B$-Mod respectively:
\begin{align*}
	KE_0(T)&=\left\{X\in A\text{-Mod}:\;\Ext_A^1(T,X)=0\right\}\\
	KE_1(T)&=\left\{X\in B\text{-Mod}:\;\Hom_A  (T,X)=0\right\}\\[1em]
	KT^0(T)&=\left\{Y\in A\text{-Mod}:\;\Tor^B_1(T,Y)=0\right\}\\
	KT^1(T)&=\left\{Y\in B\text{-Mod}:\;T\otimes_B Y=0 \right\}.
\end{align*}

Then we have the following results.
\begin{proposition}[Brenner and Butler \cite{BrennerButler80}]\label{proposition:1tilt}
	In the setting above:
	\begin{enumerate}[label=\roman*)]
		\item $(KE_0(T), KE_1(T))$ and $(KT^1(T), KT^0(T))$ are torsion pairs respectively
			in $A$-Mod and $B$-Mod.
		\item There are equivalences of (sub)categories
			\[\xymatrix@C=5pc{KE_0(T)\ar@<.2pc>[r]^-{\Hom_A(T,?)}&KT^0(T)\ar@<.2pc>[l]^-{T\otimes_B?}}\]
			\[\xymatrix@C=5pc{KE_1(T)\ar@<.2pc>[r]^-{\Ext_A^1(T,?)}&KT^1(T)\ar@<.2pc>[l]^-{\Tor_1^B(T,?)}}.\]
	\end{enumerate}
\end{proposition}

This proposition shows that the 1-tilting case is slightly more complex than the 0-tilting one.
Instead  of having an equivalence of the whole categories $A$-Mod and $B$-Mod we have two pairs of
equivalent subcategories, giving a functorial decomposition of every module in its torsion and torsion free parts. 

For an arbitrary $n\geq0$, following Miyashita, we find that every classical $n$-tilting module $T$
gives rise to two sets of $n+1$ full subcategories, of $A$-Mod and
$B$-Mod respectively, defined as follows for $e=0,\dots,n$:
\begin{align*}
	KE_e(T)&=\left\{X\in A\text{-Mod}:\;\Ext_A^i(T,X)=0\;\text{for every } i\neq e\right\}&\subset A\text{-Mod}\\
	KT^e(T)&=\left\{Y\in B\text{-Mod}:\;\Tor^B_i(T,Y)=0\;\text{for every } i\neq e\right\}&\subset B\text{-Mod}
\end{align*}
where conventionally $\Ext_A^0(T,X)=\Hom_A(T,X)$ and $\Tor^B_0(T,Y)=T\otimes_BY$.
As a generalisation of point $ii)$ of Proposition \ref{proposition:1tilt}, we may state the following result.
\begin{proposition}[Miyashita {\cite[Theorem 1.16]{Miyashita86}}]\label{proposition:miyashita-equivalences}
	In the setting above, there are equivalences of (sub)categories, for every $e=0,\dots,n$:
	\[\xymatrix@C5pc{KE_e(T)\ar@<.2pc>[r]^-{\Ext_A^e(T,?)}&KT^e(T)\ar@<.2pc>[l]^-{\Tor^B_e(T,?)}}.\]
\end{proposition}

For $n\geq2$, however, the Miyashita classes do not provide a decomposition of every module,
as it used to happen for $n=1$. This is proved by the existence of simple modules (which can have only a trivial decomposition in the module category)  not
belonging to any class.

\begin{example}[{\cite[Example 2.1]{Tonolo02}} ]\label{ex:2tilt}
Let $k$ be an algebraically closed field.
Let $A$ be the $k$-algebra associated to the quiver
$\xymatrix@1{1\ar[r]^-a&2\ar[r]^-b&3}$ with the relation $b\circ a=0$. The
indecomposable projectives are $\begin{smallmatrix}1\\2\end{smallmatrix},
	\begin{smallmatrix}2\\3\end{smallmatrix},\begin{smallmatrix}3\end{smallmatrix}$,
		while the 
indecomposable injectives are $\begin{smallmatrix}1\end{smallmatrix},
\begin{smallmatrix}1\\2\end{smallmatrix},\begin{smallmatrix}2\\3\end{smallmatrix}$.
It follows that the module
$T=\begin{smallmatrix}2\\3\end{smallmatrix}\oplus\begin{smallmatrix}1\\2\end{smallmatrix}\oplus\begin{smallmatrix}1\end{smallmatrix}$
is a classical 2-tilting module: a $p_2)$ resolution is
\[\resizebox{\textwidth}{!}{\xymatrix@C1pc{
	P^\bullet\to T\to0:&0\ar[r]& {0\oplus
		0\oplus
		\begin{smallmatrix}3\end{smallmatrix}} \ar[r]&
			0 \oplus
			0 \oplus
				{\begin{smallmatrix}2\\3\end{smallmatrix}}\ar[r]&
	{\begin{smallmatrix}2\\3\end{smallmatrix}}\oplus
		{\begin{smallmatrix}1\\2\end{smallmatrix}}\oplus
		{\begin{smallmatrix}1\\2\end{smallmatrix}}\ar[r]&
	{\begin{smallmatrix}2\\3\end{smallmatrix}}\oplus
		{\begin{smallmatrix}1\\2\end{smallmatrix}}\oplus
		{\begin{smallmatrix}1\end{smallmatrix}}\ar[r]&
			0};} \]
		$T$ is a direct sum of injectives, so it is
		rigid; lastly,
$A=\begin{smallmatrix}3\end{smallmatrix}\oplus\begin{smallmatrix}2\\3\end{smallmatrix}\oplus\begin{smallmatrix}1\\2\end{smallmatrix}$
	and so a $g_2)$ co-resolution can be easily found.
We shall show that the simple module
${\begin{smallmatrix}2\end{smallmatrix}}$ does not belong to any of the
	Miyashita classes.

In order to compute the $\Ext_A^i(T,{\begin{smallmatrix}2\end{smallmatrix}})$ we apply the
contravariant functor $\Hom_A(?,{\begin{smallmatrix}2\end{smallmatrix}})$ to the
sequence $\xymatrix@1{0\ar[r]&P^\bullet\ar[r]&0}$, obtaining
	\[\resizebox{\textwidth}{!}{\xymatrix@C1pc{
	0\ar[r]&
	\Hom_A({\begin{smallmatrix}2\\3\end{smallmatrix}}\oplus{\begin{smallmatrix}1\\2\end{smallmatrix}}\oplus
		{\begin{smallmatrix}1\\2\end{smallmatrix}}, {\begin{smallmatrix}2\end{smallmatrix}})\ar[r] &
	\Hom_A(0\oplus0\oplus{\begin{smallmatrix}2\\3\end{smallmatrix}},{\begin{smallmatrix}2\end{smallmatrix}})\ar[r]&
	\Hom_A(0\oplus0\oplus{\begin{smallmatrix}3\end{smallmatrix}},{\begin{smallmatrix}2\end{smallmatrix}})\ar[r]&
	0
		}}\]
which is isomorphic to
	\[\xymatrix{
	0\ar[r]&
	\Hom_A({\begin{smallmatrix}2\\3\end{smallmatrix}},{\begin{smallmatrix}2\end{smallmatrix}})\ar[r]^-0&
	\Hom_A({\begin{smallmatrix}2\\3\end{smallmatrix}},{\begin{smallmatrix}2\end{smallmatrix}})\ar[r]^-0&
	0\ar[r]&0
		}.\]
	Hence,
	$\Hom_A(T,{\begin{smallmatrix}2\end{smallmatrix}})\simeq\Ext_A^1(T,{\begin{smallmatrix}2\end{smallmatrix}})
		\simeq\Hom_A({\begin{smallmatrix}2\\3\end{smallmatrix}},{\begin{smallmatrix}2\end{smallmatrix}})\neq0$
	as abelian groups.
\end{example}

Indeed, those modules for which the $KE_i(T)$ (resp. the $KT^i(T)$) provide a decomposition can be characterised in the
following way.

\begin{definition}
	A left $A$-module M (resp. a left $B$-module $N$) is \emph{sequentially static}
	(resp. \emph{costatic}) if for every $i\neq j\geq 0$,
	\[\Tor^B_i(T, \Ext_A^j(T, M)) = 0 \quad(\text{resp. } \Ext_B^i(T, \Tor^A_j(T, N))=0 ).\]
\end{definition}
Notice that for an $A$-module $M$ (resp. a $B$-module $N$) to be sequentially static (resp. costatic) means that
for every $e=0,\dots,n$ we have that $\Ext_A^e(T,M)$ belongs to $KT^e$ (resp. $\Tor^B_e(T,N)$ belongs to
$KE_e$).

\begin{proposition}[{\cite[Theorem 2.3]{Tonolo02}} ]
	A left $A$-module $M$ is sequentially static if and only if there exists a filtration
	\[M=M_n\geq M_{n-1}\geq M_{n-2}\geq \cdots\geq M_0\geq M_{-1}=0\]
	such that for every $i=0,\dots,n$ the quotient $M_i/M_{i-1}$ belongs to $KE_i(T)$.
	In this case, for every such $i$ we have that $M_i/M_{i-1}\simeq \Tor^B_i(T,\Ext_A^i(T,M))$.

	Dually, a left $B$-module $M$ is sequentially costatic if and only if there exists a
	filtration
	\[N=N_{-1}\geq N_0\geq N_1\geq \cdots\geq N_{n-1}\geq N_{n}=0\]
	such that for every $i=0,\dots,n$ the quotient $N_{i-1}/N_i$ belongs to $KT^i(T)$.
	In this case, for every such $i$ we have that $N_{i-1}/N_i\simeq \Ext_A^i(T,\Tor^B_i(T,N))$.
\end{proposition}

\begin{remark}

In Example \ref{ex:2tilt} the module ${\begin{smallmatrix}2\end{smallmatrix}}$ was not sequentially
static. Let us check that 
	\[\Tor_2^B(T,\Hom_A(T,{\begin{smallmatrix}2\end{smallmatrix}}))\neq0.\]
The ring $B=\End_A(T)$ (with multiplication the composition left to right) is the $k$-algebra associated to the quiver
$\xymatrix@1{4\ar[r]^-c&5\ar[r]^-d&6}$ with the relation $d\circ c=0$.
In detail, the idempotents are the endomorphisms of $T$ induced by the identities of its direct
summands, $e_4$ of ${\begin{smallmatrix}1\end{smallmatrix}}$, $e_5$ of
${\begin{smallmatrix}1\\2\end{smallmatrix}}$ and $e_6$ of
${\begin{smallmatrix}5\\6\end{smallmatrix}}$ respectively; and $c$ and $d$
are the endomorphisms of $T$ induced by the morphisms
${\begin{smallmatrix}1\\2\end{smallmatrix}}\to{\begin{smallmatrix}1\end{smallmatrix}}$ and
${\begin{smallmatrix}2\\3\end{smallmatrix}}\to{\begin{smallmatrix}1\\2\end{smallmatrix}}$
respectively.

In order to compute the right $B$-module structure of $T$, we notice first that as a $k$-vector
	space $T$ is generated by five elements:
$x\in{\begin{smallmatrix}2\\3\end{smallmatrix}}\setminus{\begin{smallmatrix}3\end{smallmatrix}}$ and
$y=bx\in{\begin{smallmatrix}3\end{smallmatrix}}$,
$v\in{\begin{smallmatrix}1\\2\end{smallmatrix}}\setminus{\begin{smallmatrix}2\end{smallmatrix}}$ and
$w=av\in{\begin{smallmatrix}2\end{smallmatrix}}$, and $z\in{\begin{smallmatrix}1\end{smallmatrix}}$.
If we look at how $B$ acts on the right on these elements, we see that $T$ as a right $B$-module
is isomorphic to ${\begin{smallmatrix}5\\4\end{smallmatrix}}\oplus
{\begin{smallmatrix}6\\5\end{smallmatrix}}\oplus{\begin{smallmatrix}6\end{smallmatrix}}= 
{\begin{smallmatrix}v\\z\end{smallmatrix}}\oplus{\begin{smallmatrix}x\\w\end{smallmatrix}}\oplus{\begin{smallmatrix}y\end{smallmatrix}}$.
			
To compute $\Ext_A^1(T,{\begin{smallmatrix}2\end{smallmatrix}})$, we
	consider the injective coresolution of $\begin{smallmatrix}2\end{smallmatrix}$ in $A$-Mod
	\[\xymatrix{0\ar[r]&{\begin{smallmatrix}2\end{smallmatrix}}\ar[r]&
		{\begin{smallmatrix}1\\2\end{smallmatrix}}\ar[r]&{\begin{smallmatrix}1\end{smallmatrix}}\ar[r]&0}\]
			and compute $\coker\left[
				\Hom_A(T,{\begin{smallmatrix}1\\2\end{smallmatrix}})\to
				\Hom_A(T,{\begin{smallmatrix}1\end{smallmatrix}})
					\right]$ as left $B$-modules.

The left $B$-module $\Hom_A(T,{\begin{smallmatrix}1\\2\end{smallmatrix}})$ is generated as a
	$k$-vector space by (the morphisms induced on $T$ by) two morphisms
	${\begin{smallmatrix}2\\3\end{smallmatrix}}\to{\begin{smallmatrix}1\\2\end{smallmatrix}}$
		and
	${\begin{smallmatrix}1\\2\end{smallmatrix}}\to{\begin{smallmatrix}1\\2\end{smallmatrix}}$.
	When we look at how $B$ acts on the left on these elements, we see
	that the module is isomorphic to
	${}_B({\begin{smallmatrix}5\\6\end{smallmatrix}})$.
	Similarly, it can be seen that
	$\Hom_A(T,{\begin{smallmatrix}1\end{smallmatrix}})$ as a left
	$B$-module is isomorphic to
	${\begin{smallmatrix}4\\5\end{smallmatrix}}$, hence the
	cokernel we are interested in is the simple
	${\begin{smallmatrix}4\end{smallmatrix}}$.
	To compute $\Tor^B_2(T,\begin{smallmatrix}4\end{smallmatrix})$, we now consider the presentation
\[\xymatrix{
	0\ar[r]&
	{\begin{smallmatrix}5\end{smallmatrix}}\ar[r]&
	{\begin{smallmatrix}4\\5\end{smallmatrix}}\ar[r]&
	{\begin{smallmatrix}4\end{smallmatrix}}\ar[r]&0
	}\]
where ${\begin{smallmatrix}4\\5\end{smallmatrix}}$ is a projective left
	$B$-module. It can be easily seen that
	$\Tor_2^B(T,{\begin{smallmatrix}4\end{smallmatrix}})\simeq\Tor_1^B(T,{\begin{smallmatrix}5\end{smallmatrix}})$.
		Take the injective coresolution of ${}_B{\begin{smallmatrix}5\end{smallmatrix}}$
	\[\xymatrix{0\ar[r]&
		{\begin{smallmatrix}6\end{smallmatrix}}\ar[r]&
		{\begin{smallmatrix}5\\6\end{smallmatrix}}\ar[r]&
		{\begin{smallmatrix}5\end{smallmatrix}}\ar[r]&
		0};\]
similarly to what we did to compute $\Ext_A^1(T,{\begin{smallmatrix}2\end{smallmatrix}})$, 
we can compute $\Tor^B_1(T,{\begin{smallmatrix}5\end{smallmatrix}})$ as the kernel of
$T\otimes_B{\begin{smallmatrix}6\end{smallmatrix}}\to
	T\otimes_B{\begin{smallmatrix}5\\6\end{smallmatrix}}$
as a morphism of left $A$-modules.

If we call $t$ a generator of ${\begin{smallmatrix}6\end{smallmatrix}}$, with the previous notation
for the generators of $T_B$, 
as a $k$-vector space $T\otimes_B {\begin{smallmatrix}6\end{smallmatrix}}$ is generated by 
$v\otimes t, z\otimes t, x\otimes t, w\otimes t, y\otimes t$. Since however $e_6t=t$, the only
generators of these which are not zero are $x\otimes t=xe_6\otimes t$ and $y\otimes t=ye_6\otimes t$.
If we look at the action of $A$ on the left of these elements, we deduce that $T\otimes_B {\begin{smallmatrix}6\end{smallmatrix}}$
 is isomorphic to ${\begin{smallmatrix}2\\3\end{smallmatrix}}$ as a left $A$-module.
Similarly, $T\otimes_B{\begin{smallmatrix}5\\6\end{smallmatrix}}$ turns out to be isomorphic to ${\begin{smallmatrix}2\end{smallmatrix}}$,
so in the end
	\[ \Tor^B_2(T,\Hom_A(T,{\begin{smallmatrix}2\end{smallmatrix}}))
		\simeq {\begin{smallmatrix}3\end{smallmatrix}} \neq 0. \]

\end{remark}

\section{First attempts to recover the decomposition}\label{sec:2}

In order to recover a decomposition of every module induced by a classical $n$-tilting module, different strategies has been proposed.

In \cite{JensenMadsenSu13}, Jensen, Madsen and Su suggested a solution for the $n=2$ case by enlarging the
subcategories $KE_0, KE_1, KE_2$ in the following way.
Let $\mathcal{K}_0$ be the full subcategory of cokernels  of monomorphisms from objects in $KE_2$ to objects
in $KE_0$; let $\mathcal{K}_1$ be $KE_1$; let $\mathcal{K}_2$ be the full subcategory of kernels of epimorphisms from
objects in $KE_2$ to objects in $KE_0$:
\begin{align*}
	\mathcal{K}_0 &= \left\{ \coker f:\;X_2\stackrel{f}{\hookrightarrow} X_0,\quad X_2\in KE_2, X_0\in KE_0\right\}\\
	\mathcal{K}_1 &= KE_1\\
	\mathcal{K}_2 &= \left\{ \ker g:\;X_2\stackrel{g}{\twoheadrightarrow} X_0,\quad X_2\in KE_2, X_0\in KE_0\right\}.
\end{align*}
By considering the morphisms $f:\;0\hookrightarrow X_0$ and $g:\;X_2\twoheadrightarrow 0$ we can see that $KE_i\subset
\mathcal{K}_i$ for every $i=0,1,2$, so this is indeed an enlargement.

Now, for $i=0,1,2$ let $\mathcal{E}_i$ be the extension closure of $\mathcal{K}_i$, i.e. the smallest
subcategory containing $\mathcal{K}_i$ and closed under extensions.

\begin{proposition}[{\cite[Corollary 15, Theorem 19, Lemma 24]{JensenMadsenSu13}}]\label{proposition:JMS2}
	For any left $A$-module $X$ there exists a unique filtration
	\[0=X_0\subseteq X_1\subseteq X_2\subseteq X_3=X\]
	with the quotients $X_{i+1}/X_i\in\mathcal{E}_i$ for every $i=0,1,2$. Moreover, such a
	filtration is functorial.
\end{proposition}

\begin{example} \label{ex:JMS}
	Let us apply this construction to find a decomposition of the simple module $\begin{smallmatrix}2\end{smallmatrix}$ considered in the Example \ref{ex:2tilt}.
	In a way similar to that used to study the $\Ext_A^i(T,{\begin{smallmatrix}2\end{smallmatrix}})$, 
	$i=0,1,2$, we may prove
	that ${\begin{smallmatrix}2\\3\end{smallmatrix}}$ belongs to $KE_0$ and
	${\begin{smallmatrix}3\end{smallmatrix}}$ belongs to $KE_2$. Then,
	${\begin{smallmatrix}2\end{smallmatrix}}$
	belongs to $\mathcal{K}_0\subseteq\mathcal{E}_0$, being the cokernel of the monomorphism
	${\begin{smallmatrix}3\end{smallmatrix}}\to{\begin{smallmatrix}2\\3\end{smallmatrix}}$.
	Therefore the trivial filtration $0\leq \begin{smallmatrix}2\end{smallmatrix}$ has its only
	filtration factor in the new class $\mathcal{E}_0$.
\end{example}

In \cite{Lo15}, Lo generalised this filtration to the $n>2$ case as well. After giving a different
proof of Proposition \ref{proposition:JMS2}, he introduced the following
subcategories. For a class of objects $\mathcal{S}$, denote by $\left[\mathcal{S}\right]$ the
extension closure of the full subcategory generated by quotients of objects of $\mathcal{S}$:
\[ \left[\mathcal{S}\right] = \langle \left\{X:\;\exists( S\twoheadrightarrow X)\text{ for some
}S\in \mathcal{S}\right\}\rangle_{\text{ext}}\;. \]
This subcategory is closed under quotients (\cite[Lemma 5.1]{Lo15}).
Then set, for $i=0,\dots,n$:
\begin{align*}
	\mathcal{T}_i &= \left[\Ker \Ext_A^i(T,?)\cap\dots\cap\Ker\Ext_A^n(T,?)\right]\\
	\mathcal{F}_i &= \Ker\Hom_A(\mathcal{T}_i,?) = \left\{ X:\; \Hom_A(\mathcal{T}_i,X)=0 \right\}
\end{align*}
with our usual convention that $\Ext_A^0=\Hom_A$. Define also $\mathcal{T}_{n+1}=A\text{-Mod}$ and
$\mathcal{F}_{n+1}=0$.

This provides pairs $(\mathcal{T}_i,\mathcal{F}_i)$ of full subcategories, which are torsion pairs
since the $\mathcal{T}_i$'s are closed under extensions and quotients (see Polishchuk \cite{Polishchuk07}).
The following easy proposition can then be applied to these torsion pairs.

\begin{proposition}[{\cite[ Theorem 5.3]{Lo15}}]\label{proposition:Lo}
		Let $(\mathcal{T}_i,\mathcal{F}_i)$ be torsion pairs in $A$-Mod, for
		$i=0,\dots,n+1$, such that
		\[ 0=\mathcal{T}_0\subseteq\mathcal{T}_1\subseteq\dots\subseteq\mathcal{T}_{n+1}=A\text{-Mod}. \]
		Then for every left $A$-module $X$ there exists a functorial filtration
		\[ 0=X_0\subseteq X_1\subseteq \dots\subseteq X_{n+1}=X\]
		such that $X_i\in\mathcal{T}_i$ for $i=0,\dots,n+1$ and
		$X_i/X_{i-1}\in\mathcal{T}_{i}\cap\mathcal{F}_{i-1}$ for $i=1,\dots,n+1$.
		Moreover, the $\mathcal{T}_i\cap\mathcal{F}_{i-1}$ have pairwise trivial
		intersection.
\end{proposition}

We now prove that the subcategories $\mathcal{T}_{i}\cap\mathcal{F}_{i-1}$ are indeed enlargments of
the Miyashita classes using the following  generalisation of \cite[Lemma 3.2]{Bazzoni04}, which we
find of independent interest.

\begin{lemma}\label{lemma:description}
	Let $X$ be a module belonging to $\cap_{i>e}\Ker\Ext_A^i(T,X)$ for some $0\leq~e\leq
	n$. Then, there exists a sequence of direct summands of coproducts of copies of $T$,
	\[\xymatrix{\cdots\ar[r]&T_{-1}\ar[r]^-{d_{-1}}&T_0\ar[r]^-{d_0}&\cdots\ar[r]&T_e\ar[r]&0}\]
	exact everywhere except for degree 0, and having $\ker d_0/\im d_{-1}\simeq X$.
	In particular, for $e=n$, $\cap_{i>n}\Ker\Ext_A^i(T,X)=A\text{-Mod}$ and hence $X$ may be any module.
\end{lemma}
\begin{proof}Set $T^{\bot_\infty}:=\cap_{i>0} \Ker \Ext^i(T,?)$ and, for a family of
modules $\mathcal{S}$, ${}^\bot\mathcal{S}: = \Ker \Ext^1(?,\mathcal{S})$. It is
well known (see \cite{GoebelTrlifaj}, after Definition 5.1.1) that the pair of
subcategories $({}^{\bot}(T^{\bot_\infty}),T^{\bot_\infty})$ is a complete hereditary
cotorsion pair. This means (see \cite[Lemma 2.2.6]{GoebelTrlifaj}) that $X$ (as any other module) admits a special
	${}^\bot(T^{\bot_\infty})$-precover
	\[ \xymatrix{0\ar[r]&J\ar[r]&K\ar[r]&X\ar[r]&0}. \]
In particular, $J$ belongs to
$({}^\bot(T^{\bot_\infty}))^\bot$, which equals $T^{\bot_\infty}$ by definition of cotorsion pair.
	Now we can apply \cite[Lemma 3.2]{Bazzoni04} to $J$ and \cite[Proposition 5.1.9]{GoebelTrlifaj}
	to $K$ in order to construct a sequence of direct summands of coproducts of copies of $T$
	\[\xymatrix{
		\cdots\ar[r]& T_{-2}
		\ar[r]&T_{-1}\ar@{->>}[d]\ar@{..>}[r]^-{d_{-1}}&T_0\ar[r]^-{d_0}&\cdots\ar[r]&T_n\ar[r]&0&(\ast)\\
		&& J \ar@{^{(}->}[r]& K \ar@{^{(}->}[u]\ar@{->>}[r]^-\pi&X}
	\]
By construction the first row is a sequence which is exact everywhere except for degree 0,
	where $\ker d_0/\im d_{-1} \simeq K/J \simeq X$.

	This concludes the proof for the case where $e=n$. Otherwise,
	it can be easily proved that since by hypothesis $\Ext_A^i(T,X)=0$ for $i>e$, then for these
	indices $T_i=0$: let us show it for $i=n$, then the other cases follow similarly.
First, notice that since $\Ext_A^j(T,J)=0$, one gets $\Ext_A^j(T,K)=\Ext_A^j(T,X)$ for every $j>0$ .
	Then, if we call $K_j = \ker d_j$ for $j\geq0$ (and so $K_0=K$), we have
\[\Ext^1_A(T,K_{n-1})\cong  \Ext^n_A(T,K_{0})=\Ext^n_A(T,X)=0;\]
applying the functor $\Hom_A(T,?)$ to the short exact sequence
\[0\to K_{n-1}\to T_{n-1}\to K_n=T_n\to 0\]
we get that $\Hom_A(T, T_{n-1})\to\Hom_A(T,T_n)$ is an epimorphism and hence all morphisms $T\to T_n$ factorise through $T_{n-1}$. Using
	the universal property of the coproduct of which $T_n$ is a direct summand, it is easy to
	prove that this implies that $0\to K_{n-1}\to T_{n-1}\to T_n\to 0$ splits.
%
%
Thus $K_{n-1}$ is a direct summand of a
	coproduct of copies of $T$. Therefore, we may truncate the sequence $(\ast)$ as
	\[\xymatrix@C1pc{\cdots\ar[r]&T_{n-3}\ar[r]&T_{n-2}\ar[r]&K_{n-1}\ar[r]&0}.\]
\end{proof}

	Notice that this lemma generalises \cite[Lemma 3.2]{Bazzoni04}, which is the case where
	$e=0$.

\begin{remark}\label{remark:lo}
We shall prove that $KE_e \subseteq \mathcal{T}_{e+1}\cap\mathcal{F}_e$ for every $e=0,\dots,n$.
 Indeed, it is obvious that
$KE_e\subseteq\mathcal{T}_{e+1}$. To see that any $M\in KE_e$ belongs to $\mathcal{F}_e$ as well,
we will proceed in subsequent steps.

First, we prove that for every $X\in\cap_{i>e-1} \Ker\Ext_A^i(T,?)\subseteq\mathcal{T}_e$ there are
no non zero morphisms $X\to M$. Indeed, if $e=0$ then $X=0$; if $e>0$ consider the sequence 
\[T^\bullet:=\xymatrix{\cdots\ar[r]&T_{-1}\ar[r]^-{d_{-1}}&T_0\ar[r]^-{d_0}&\cdots\ar[r]&T_{e-1}\ar[r]&0}\]
given by Lemma
\ref{lemma:description} applied to $X$. Set $K_j = \ker d_j$ for $j\geq0$ (and so $K_0=K$), applying the functor $\Hom(-,M)$ to the epimorphism $K_0\to X$, one gets
\[\Hom_A(X,M)\hookrightarrow \Hom_A(K_0, M) \cong \Ext^1_A(K_1,M)\cong\cdots\]
\[\cdots \cong \Ext^{e-1}_A(K_{e-1},M)=
\Ext^{e-1}_A(T_{e-1},M)=0,\]
and hence $\Hom_A(X,M)=0$.
%
%
%
%
%

Now, if $X'$ is the epimorphic image of some $X\in\cap_{i>e-1}\Ker\Ext_A^i(T,?)$,
	we have $\Hom_A(X',M)\hookrightarrow\Hom_A(X,M)=0$ so $\Hom_A(X',M)=0$ as well.
	Lastly, if $X''$ is an extension of such epimorphic images, we still find that
	$\Hom_A(X'',M)=0$.
	
	This proves the claim that $M$ has no non zero morphisms from objects of
	$\mathcal{T}_e$, and therefore it belongs to $\mathcal{F}_e$.
\end{remark}

The last result of \cite{Lo15} is the proof that for $n=2$ the filtration procedure of Proposition
\ref{proposition:Lo} reduces to that provided
by Jensen, Madsen and Su.

It should be noted that these results, while providing a way to generalise the decomposition of
every module found in the $n=1$ case, do so by introducing enlargements of the Miyashita classes
$KE_i$ which are not very natural, at the point that the connection to the tilting object they
originate from seems a bit weak.

The rest of the article is devoted to the description of an alternative approach to this enlarging
strategy, introduced in \cite{FiorotMattielloTonolo16}, which takes
place in the derived category $\mathcal{D}(A)$ of $A$-Mod.
In the following section we recall some basic facts about
derived categories and $t$-structures.

\section{Introducing derived categories and $t$-structures}\label{sec:3}

Given an abelian category $\mathcal{A}$, one may construct its derived category $\mathcal{D}(\mathcal{A})$ defining
objects and morphisms in the following way. As objects, one takes the cochain complexes with terms
in $\mathcal{A}$:
\[ \xymatrix{\cdots\ar[r]&X^n\ar[r]^-{d_X^n}&X^{n+1}\ar[r]^-{d_X^{n+1}}&X^{n+2}\ar[r]&\cdots} \]
In order to define morphisms, one first takes the quotient of morphisms of
complexes modulo those satisfying the \emph{nullohomotopy} condition; the category having these
equivalence classes as morphisms is called the \emph{homotopy category}. The step from this to the
derived category is performed by an argument of \emph{localisation}; in this way, morphisms of
complexes which induce isomorphisms on the cohomologies get an inverse in the derived category.

The category $\mathcal{D}(\mathcal{A})$ so obtained is not abelian anymore, but it is a
\emph{triangulated category}. This means that it is equipped with the following structure. First, there is
an autoequivalence, whose action on the complex $X^\bullet$ is denoted as $X^\bullet[1]$ and
is defined as follows:
\[ (X^\bullet[1])^n=X^{n+1}\qquad d_{X[1]}^n=-d_X^{n+1}. \]
This functor is called the \emph{suspension functor}; its natural definition on chain morphisms
induces a good definition on morphisms in $\mathcal{D}(\mathcal{A})$. We will sometimes denote this
functor also as $\Sigma$; its inverse as $\Sigma^{-1}$ or $?[-1]$; their powers as $\Sigma^i$ or
$?[i]$ for $i\in\mathbb{Z}$.

Given this autoequivalence, one calls \emph{triangles} the diagrams of the form
\[\xymatrix{X^\bullet\ar[r]^-u&Y^\bullet\ar[r]^-v&Z^\bullet\ar[r]^-w&X[1]}\]
such that $v\circ u=0=w\circ v$; in $\mathcal{D}(\mathcal{A})$ a particular role is played by the
triangles isomorphic (as diagrams) to those of the form
\[\xymatrix{X^\bullet\ar[r]^-f&Y^\bullet\ar[r]&\Cone f\ar[r]& X[1]}\]
where $\Cone f$ is defined as the complex having terms
$(\Cone f)^i=X^{i+1}\oplus Y^i$ and differentials
$d_{\Cone f}^i = \left[\begin{smallmatrix}-d_X^{i+1}&0\\f^{i+1}&d_Y^i\end{smallmatrix}\right]$.
These triangles are called \emph{distinguished triangles} and are the analogous of short exact
sequences in abelian categories.

In a triangulated category, hence also in $\mathcal{D}(\mathcal{A})$, products and coproducts of distinguished
triangles, when they exist, are distinguished (see \cite[Proposition 1.2.1, and its
dual]{Neeman01}). In particular, if $\mathcal{A}$ has arbitrary products or coproducts,
$\mathcal{D}(\mathcal{A})$ has them as well: they are constructed degree-wise using those of
$\mathcal{A}$.

Once we have set our context, we now define the main object which we will work with.

\begin{definition}
	Let $\mathcal{S}=(\mathcal{S}^{\leq 0},\mathcal{S}^{\geq 0})$ be a pair of full,
	strict (i.e. closed under isomorphisms) subcategories of $\mathcal{D}(\mathcal{A})$, and
	denote $\mathcal{S}^{\leq i}=\mathcal{S}^{\leq 0}[-i]$ and $\mathcal{S}^{\geq
	i}=\mathcal{S}^{\geq 0}[-i]$, for every $i\in\mathbb{Z}$.

	The pair $\mathcal{S}$ is a \emph{$t$-structure} if it satisfies the following properties:
	\begin{enumerate}[label=T\arabic*)]
		\item $\mathcal{S}^{\leq 0}\subseteq\mathcal{S}^{\leq 1}$ and
			$\mathcal{S}^{\geq 0}\supseteq\mathcal{S}^{\geq 1}$;
		\item $\Hom_{\mathcal{D}(\mathcal{A})}(\mathcal{S}^{\leq0},\mathcal{S}^{\geq1})=0$;
		\item For any complex $X^\bullet$ in $\mathcal{D}(\mathcal{A})$, there exist complexes
			$A^\bullet \in\mathcal{S}^{\leq0}$ and $B^\bullet\in\mathcal{S}^{\geq1}$
			and morphisms such that
			\[\xymatrix{A^\bullet\ar[r]&X^\bullet\ar[r]&B^\bullet\ar[r]&A^\bullet[1]}\]
			is a distinguished triangle in $\mathcal{D}(\mathcal{A})$. This is called an
			\emph{approximating triangle} of $X^\bullet$.
	\end{enumerate}
	In this case, $\mathcal{S}^{\leq 0}$ is called an \emph{aisle}, $\mathcal{S}^{\geq 0}$ a
	\emph{coaisle}.
	The $t$-structure $\mathcal{S}$ is called \emph{non degenerate} if
	$\bigcap_{i\in\mathbb{Z}}\mathcal{S}^{\leq i}=0$ (or equivalently
	$\bigcap_{i\in\mathbb{Z}}\mathcal{S}^{\geq i}=0$).
	The full subcategory $\mathcal{H}_\mathcal{S}=\mathcal{S}^{\leq
	0}\cap\mathcal{S}^{\geq 0}$ is called the \emph{heart} of $\mathcal{S}$.
\end{definition}

This definition immediately resembles that of a torsion pair in an abelian category. As it holds
for torsion pairs, the approximating triangle of a complex with respect to a $t$-structure is
functorial, as we are going to state.

Given a $t$-structure $\mathcal{S}$ in $\mathcal{D}(\mathcal{A})$, it can be proved that the embeddings of subcategories
$\mathcal{S}^{\leq0}\subseteq \mathcal{D}(\mathcal{A})$ and
$\mathcal{S}^{\geq0}\subseteq \mathcal{D}(\mathcal{A})$
have a right adjoint
$\sigma^{\leq0}:\;\mathcal{D}(\mathcal{A})\to\mathcal{S}^{\leq0}$
and a left adjoint
$\sigma^{\geq0}:\;\mathcal{D}(\mathcal{A})\to\mathcal{S}^{\geq0}$
respectively.

For $i\in\mathbb{Z}$, let us write $\sigma^{\leq i}=\Sigma^{-i}\circ\sigma^{\leq0}\circ
\Sigma^i:\;\mathcal{D}(\mathcal{A})\to\mathcal{S}^{\leq i}$ and
similarly $\sigma^{\geq i}=\Sigma^{-i}\circ\sigma^{\geq
0}\circ\Sigma^i:\;\mathcal{D}(\mathcal{A})\to\mathcal{S}^{\geq i}$;
$\sigma^{\leq i}$ and $\sigma^{\geq i}$ will be called respectively the left and the right
\emph{truncation functors at $i$} with respect to $\mathcal{S}$, for $i\in\mathbb{Z}$.

It can be proved that for every $X^\bullet$ in $\mathcal{D}(\mathcal{A})$, the
approximation triangle for $X^\bullet$ provided by the definition of the $t$-structure $\mathcal{S}$
is precisely (isomorphic to):
\[ \xymatrix{ \sigma^{\leq0}(X^\bullet)\ar[r]&X^\bullet\ar[r]&\sigma^{\geq1}
(X^\bullet)\ar[r]&(\sigma^{\leq0}(X^\bullet))[1]}.\]

The truncation functors of $\mathcal{S}$ can be used to define the \emph{$i$-th cohomology} with respect to
$\mathcal{S}$. 
It can be proved that for every $i,j\in\mathbb{Z}$ there is a canonical
natural isomorphism $\sigma^{\leq i}\sigma^{\geq j}\simeq \sigma^{\geq j}\sigma^{\leq i}$.
Then, for every $i\in\mathbb{Z}$, the functor $H_\mathcal{S}^i=\Sigma^i\sigma^{\leq i}\sigma^{\geq
i}\simeq \Sigma^i\sigma^{\geq i}\sigma^{\leq i}:\;
\mathcal{D}(\mathcal{A})\to\mathcal{H}_\mathcal{S}$ is called the $i$-th \emph{cohomology functor}
with respect to the $t$-structure $\mathcal{S}$ (or simply \emph{$\mathcal{S}$-cohomology}).

We introduce now the first $t$-structure in $\mathcal{D}(\mathcal{A})$ we are going to use.
\begin{definition}\label{definition:t-str:natural}
	The \emph{natural} $t$-structure $\mathcal{D}$ of $\mathcal{D}(\mathcal{A})$ has aisle and coaisle:
	\begin{align*}
		\mathcal{D}^{\leq0} &= \left\{X^\bullet\in\mathcal{D}(\mathcal{A}):\;
		H^i(X^\bullet)=0\text{ for every }i>0\right\}\\
		\mathcal{D}^{\geq0} &= \left\{X^\bullet\in\mathcal{D}(\mathcal{A}):\;
		H^i(X^\bullet)=0\text{ for every }i<0\right\}.
	\end{align*}
\end{definition}
Notice that by construction the $i$-th $\mathcal{D}$-cohomology of $X^\bullet$ is a complex
having zero cohomology everywhere except for degree 0, where it has $H^i(X^\bullet)$, the
usual $i$-th cohomology of $X^\bullet$: i.e., $H_\mathcal{D}^i(X^\bullet)= H^i(X^\bullet)[0]$.

The original proof that this is indeed a $t$-structure can be found in \cite{BBD82}.

We now state the following fundamental theorem about $t$-structures. One may read it
with our example $\mathcal{D}$ in mind.

\begin{theorem}\label{theorem:t-str:fundamental}
	Let $\mathcal{S}$ be a non degenerate $t$-structure in $\mathcal{D}(\mathcal{A})$. Then:
	\begin{enumerate}
		\item The heart $\mathcal{H}_\mathcal{S}$ is an abelian category; moreover, a short
			sequence
			\[\xymatrix{0\ar[r]&X^\bullet\ar[r]&Y^\bullet\ar[r]&Z^\bullet\ar[r]&0}\]
			in $\mathcal{H}_\mathcal{S}$ is exact if and only if there exists a morphism
			$Z\to X[1]$ in $\mathcal{D}(\mathcal{A})$ such that the triangle
			\[\xymatrix{X^\bullet\ar[r]&Y^\bullet\ar[r]&Z^\bullet\ar[r]&X^\bullet[1]}\]
			is distinguished.
		\item Given any distinguished triangle
			\[\xymatrix{X^\bullet\ar[r]&Y^\bullet\ar[r]&Z^\bullet\ar[r]&X^\bullet[1]}\]
			in $\mathcal{D}(\mathcal{A})$, there is a long exact sequence in
			$\mathcal{H}_\mathcal{S}$
			\[\resizebox{.9\textwidth}{!}{
				\xymatrix{\cdots\ar[r]&H_\mathcal{S}^{i-1}Z^\bullet\ar[r]&
				H_\mathcal{S}^iX^\bullet\ar[r]&H_\mathcal{S}^iY^\bullet\ar[r]&
					H_\mathcal{S}^iZ^\bullet\ar[r]&
				H_\mathcal{S}^{i+1}\ar[r]&\cdots}}\]
	\end{enumerate}
\end{theorem}

As can be easily seen, the heart $\mathcal{H}_\mathcal{D}$ of the natural $t$-structure of
$\mathcal{D}(\mathcal{A})$ is (equivalent to) $\mathcal{A}$ itself via the embedding
$\mathcal{A}\to\mathcal{D}(\mathcal{A})$ defined by
\[ X\mapsto X[0]=(\xymatrix@1{\cdots\ar[r]&0\ar[r]&X\ar[r]&0\ar[r]&\cdots})\]
whose quasi-inverse is $H^0$, the usual $0$th-cohomology functor.

As it happens for torsion pairs, the aisle or the coaisle of a $t$-structure is sufficient to characterise the whole $t$-structure. Indeed,
we give the following lemma by Keller and Vossieck \cite{KellerVossieck88}. 

\begin{lemma}\label{lemma:t-str:aisle-coaisle-duality}
	Let $\mathcal{R}=(\mathcal{R}^{\leq0},\mathcal{R}^{\geq0})$ be a $t$-structure in
	$\mathcal{D}(\mathcal{A})$. Then
	\begin{align*}
		\mathcal{R}^{\leq0} &=
		\left\{X^\bullet\in\mathcal{D}(\mathcal{A}):\;\Hom_{\mathcal{D}(\mathcal{A})}(X^\bullet,Y^\bullet)=0
			\text{ for all } Y^\bullet\in\mathcal{R}^{\geq1} \right\}\\
		\mathcal{R}^{\geq0} &=
		\left\{Y^\bullet\in\mathcal{D}(\mathcal{A}):\;\Hom_{\mathcal{D}(\mathcal{A})}(X^\bullet,Y^\bullet)=0
			\text{ for all } X^\bullet\in\mathcal{R}^{\leq-1} \right\}.
	\end{align*}
\end{lemma}

We now give the following proposition, which gives a very useful way to construct $t$-structures.

\begin{proposition}[{\cite[Lemma 3.1, Proposition 3.2]{TarrioLopezSalorio03}}]\label{prop:t-str:susp}
	Let $E$ be any complex in $\mathcal{D}(\mathcal{A})$, and let $\mathcal{U}$ be the smallest
	cocomplete pre-aisle containing $E$, that is: the smallest full, strict subcategory of
	$\mathcal{D}(A)$ closed under positive shifts, extensions and coproducts. Then,
	$\mathcal{U}$ is an aisle and the corresponding coaisle is
	\begin{align*}
		\mathcal{U}^\bot &=
		\left\{Y^\bullet\in\mathcal{D}(\mathcal{A}):\;\Hom_{\mathcal{D}(\mathcal{A})}(X^\bullet,Y^\bullet)=0
		\text{ for every } X^\bullet\in\mathcal{U}[1] \right\}\\
		&=
		\left\{Y^\bullet\in\mathcal{D}(\mathcal{A}):\;\Hom_{\mathcal{D}(\mathcal{A})}(E,Y^\bullet[i])=0
		\text{ for every } i<0 \right\}.
	\end{align*}
\end{proposition}

\begin{remark}\label{rmk:d-gen}
	As a first application of this proposition, it is easy to see that if $\mathcal{A}$ has a
	projective generator $E$, then the natural $t$-structure of $\mathcal{D}(\mathcal{A})$ will be
	that generated by $E$ (in the sense of the proposition). This will be the case when we will
	consider $\mathcal{A}=A\text{-Mod}$, with $E=A$.
\end{remark}

\begin{remark}\label{rmk:susp-characterisation}
	In the case where the object $E$ is in fact a module, that is, a complex concentrated in
	degree zero, we shall give a characterisation of the aisle $\mathcal{U}$ generated by $E$.

	First, $\mathcal{U}$ contains $E$; and it is closed under positive shifts, hence it contains
	$E[i]$ for
	every $i>0$. $\mathcal{U}$ is closed under arbitrary coproducts; let then $J=\cup_{i>0} J_i$
	be a set of indeces, and let $E_j=E[i]$ for every $j\in J_i$. Then the coproduct
	$\coprod_{j\in J} E_j = \coprod_{i>0} E^{(J_i)}[i]$ belongs to $\mathcal{U}$ as well.
	If $\mathcal{V}$ is the full subcategory of all objects isomorphic to these coproducts, this means that
	$\mathcal{V}\subseteq \mathcal{U}$. Since $\mathcal{U}$ is also closed under extensions, if
	we call $\mathcal{V}'$ the extension closure of $\mathcal{V}$, we have
	$\mathcal{V}'\subseteq\mathcal{U}$ as well. Moreover, since coproducts of distinguished
	triangles are distinguished, from the fact that $\mathcal{V}$ is closed under arbitrary
	coproducts follows easily that $\mathcal{V}'$ is as well. Hence, $\mathcal{V}'$ is a
	cocomplete pre-aisle, and by definition $\mathcal{U}\subseteq\mathcal{V}'$.

	In conclusion, objects of $\mathcal{U}$ are isomorphic to complexes having zero terms in
	positive degrees and coproducts of $E$ in nonpositive degrees.
\end{remark}

\section{$n$-Tilting objects and associated $t$-structures}\label{sec:4}

In the following, we are going to work with a generalisation of classical $n$-tilting modules; the
definition we give is more oriented towards the derived category $\mathcal{D}(A)$ of $A$-Mod, which
will be our setting.

\begin{definition}\label{definition:ntil:tilting-object}
	A left $A$-module $T$ is (non necessarily classical) \emph{$n$-tilting} if it satisfies the following
	properties:
	\begin{enumerate}
		\item[$P_n$)] $T$ has projective dimensions at most $n$, i.e. there exists an exact sequence
			\[\xymatrix{0\ar[r]&P_n\ar[r]&\cdots\ar[r]&P_1\ar[r]&P_0\ar[r]&T\ar[r]&0}\]
			in $A$-Mod with the $P_i$ projectives;
		\item[$E_n$)] $T$ is rigid, i.e. $\Ext_A^i(T,T^{(\Lambda)})=0$ for every index $0<i\leq n$
			and set $\Lambda$;
		\item[$G_n$)] $T$ is a \emph{generator} in $\mathcal{D}(A)$, meaning that if for a complex
			$X^\bullet$ we have $\Hom_{\mathcal{D}(A)}(T,X[i])=0$ for every
			$i\in\mathbb{Z}$, then $X^\bullet =0$ in $\mathcal{D}(A)$.
	\end{enumerate}
\end{definition}

Notice that a classical $n$-tilting module is indeed $n$-tilting: in particular, $p_n$) implies
$P_n)$, $p_n$) and $e_n$) imply $E_n$) (see the Stacks Project \cite[Proposition 15.72.3]{stacks-project}
and $g_n$) implies $G_n$) (see Positselski and Stovicek \cite[Corollary 2.6]{PositselskiStovicek17}).

The discussion about $t$-structures in the previous section is justified by the following
construction.
Let $T$ be a $n$-tilting left $A$-module and consider the pair
$\mathcal{T}=(\mathcal{T}^{\leq0},\mathcal{T}^{\geq0})$ of subcategories
of $\mathcal{D}(A)$
\begin{align*}
	\mathcal{T}^{\leq0} &= \left\{X^\bullet\in\mathcal{D}(A):\;
	\Hom_{\mathcal{D}(A)}(T,X^\bullet[i])=0\text{ for every }i>0\right\}\\
	\mathcal{T}^{\geq0} &= \left\{X^\bullet\in\mathcal{D}(A):\;
	\Hom_{\mathcal{D}(A)}(T,X^\bullet[i])=0\text{ for every }i<0\right\}.
\end{align*}

\begin{remark}\label{rmk:T-gen}
	This is the $t$-structure generated by $T$ in the sense of Proposition
	\ref{prop:t-str:susp}. Indeed, let $\mathcal{G}=(\mathcal{G}^{\leq0},\mathcal{G}^{\geq0})$ be the
	generated $t$-structure. We have $\mathcal{T}^{\geq0}=\mathcal{G}^{\geq0}$.
	For the aisle, notice that $\mathcal{T}^{\leq0}$ contains $T$ by $E_n$); and it is clearly closed
	under positive shifts, hence it contains any $T[i]$ for $i>0$. Now, we show that it is
	closed under arbitrary coproducts of such complexes $T[i]$. Let $J=\cup_{i>0}J_i$ be a set,
	let $T_j = T[i]$ for every $j\in J_i$,
	and consider the coproduct $\coprod_{j\in J} T_j = \coprod_{i>0} T^{(J_i)}[i]$. Notice that
	since by $P_n$) $T$ has projective dimension $n$,we have
	\[\Hom_{\mathcal{D}(A)}(T,\coprod_{j\in J} T_j) =
	\Hom_{\mathcal{D}(A)}(T,\coprod_{i>0}T^{(J_i)}[i]) =
	\Hom_{\mathcal{D}(A)}(T,\coprod_{1\leq i \leq n}T^{(J_i)}[i])\]
	Now, since $\mathcal{D}(A)$ is an additive category, this is itself isomorphic to
	\[\Hom_{\mathcal{D}(A)}(T,\prod_{1\leq i \leq n}T^{(J_i)}[i])\simeq
	\prod_{1\leq i\leq n} \Hom_{\mathcal{D}(A)}(T, T^{(J_i)}[i]) = 0 \]
	which is zero by property $E_n$).
	Lastly, $\mathcal{T}^{\leq0}$ is clearly closed under extensions, and so by Remark
	\ref{rmk:susp-characterisation} it contains $\mathcal{G}^{\leq0}$.

	For the inclusion $\mathcal{T}^{\leq0}\subseteq\mathcal{G}^{\leq0}$, take an
	object $X^\bullet\in\mathcal{T}^{\leq0}$, and consider its approximation triangle with
	respect to $\mathcal{G}$,
	\[\xymatrix{A^\bullet\ar[r]&X^\bullet\ar[r]&B^\bullet\ar[r]^-{+1}&}.\]
	We have $A^\bullet\in\mathcal{G}^{\leq0}\subseteq\mathcal{T}^{\leq0}$; and since
	$\mathcal{T}^{\leq0}$ is clearly closed under cones, $B^\bullet\in\mathcal{T}^{\leq0}$ as
	well. So in the end
	$B^\bullet\in\mathcal{T}^{\leq0}\cap\mathcal{G}^{\geq1}=\mathcal{T}^{\leq0}\cap\mathcal{T}^{\geq1}$
	which is 0 by $G_3$).
\end{remark}

As a side note, observe that if $T$ is classical $n$-tilting, it induces a triangulated equivalence
$R\Hom_A(T,?):\;\mathcal{D}(A)\to\mathcal{D}(B)$ (see \cite{ClineParshallScott86}); then, by the fact that
\[\Hom_{\mathcal{D}(A)}(T,X^\bullet[i])=H^iR\Hom_A(T,X^\bullet)\]
we may recognise in $\mathcal{T}:=(\mathcal{T}^{\leq0},\mathcal{T}^{\geq0})$ the ``pullback'' of the natural $t$-structure of
$\mathcal{D}(B)$ along $R\Hom_A(T,?)$.

\begin{remark}\label{rmk:review_previous}
	Without further study, the $t$-structure $\mathcal{T}$ can be immediately used to
	review some previous results.

	First, we can greatly simplify the proof of Remark \ref{remark:lo}.
	In the notation used there, to prove that there are no non zero morphisms $X\to M$, for
	$X$ in $\cap_{i> e-1}\Ker\Ext_A^i(T,?)$ and $M$ in $KE_e$, we may just recognise that
	we have $X=X[0]\in \mathcal{T}^{\leq e-1}$ and $M\in\mathcal{T}^{\geq e}$, and use
	axiom T2) of $t$-structures. 

	Second, we may read our Lemma
	\ref{lemma:description} under a different light: given the characterisation of objects in
	$\mathcal{T}^{\leq0}$ as in Remark \ref{rmk:susp-characterisation}, the lemma can be seen to be the
	equality
	\[\cap_{i>e}\Ker\Ext_A^i(T,?) = A\text{-Mod}\cap\mathcal{T}^{\leq e}.\]
\end{remark}

In the following, $T$ will be a $n$-tilting module; $\mathcal{T}$ will be the associated
$t$-structure, as defined above.
The solution that we are going to give to our decomposition problem originates from the
interaction of the $t$-structure $\mathcal{T}$ with the natural one of $\mathcal{D}(A)$ (see
Definition \ref{definition:t-str:natural}).
First, we make an easy observation.

\begin{proposition}\label{proposition:ntil:inclusions}
	The following inclusions of aisles and coaisles hold:
	\[\mathcal{D}^{\leq -n}\subseteq \mathcal{T}^{\leq0}\subseteq \mathcal{D}^{\leq0}
	\qquad \text{and}\qquad
	  \mathcal{D}^{\geq0}\subseteq \mathcal{T}^{\geq0}\subseteq \mathcal{D}^{\geq -n}. \]
\end{proposition}
\begin{proof}
	Some of the inclusions are easy to prove: if $X^\bullet \in\mathcal{D}^{\leq -n}$, then
	for every $i>0$ we will have
	$X^\bullet[i]\in\mathcal{D}^{\leq-n-i}\subseteq\mathcal{D}^{\leq -n-1}$,
	hence $\Hom_{\mathcal{D}(A)}(T,X^\bullet[i])=0$ since $T$ has projective
	dimension $n$. On the other hand, if $X^\bullet\in\mathcal{D}^{\geq 0}$, then
	for every $i<0$ we will have $X^\bullet[i]\in\mathcal{D}^{\geq0-i}\subseteq\mathcal{D}^{\geq1}$,
	hence again $\Hom_{\mathcal{D}(A)}(T, X^\bullet[i])=0$ since
	$T\in\mathcal{D}^{\leq0}$. The other two inclusions can be easily proved from these using
	Lemma \ref{lemma:t-str:aisle-coaisle-duality}.
\end{proof}

\begin{remark}\label{remark:MiyashitaHearts}
	With Proposition \ref{proposition:ntil:inclusions}, we are ready to notice an important
	fact, which will be key later. Take a
	module $X$ in $KE_e$, for some $e=0,\dots,n$; in particular, being a module, it belongs to
	$A\text{-Mod}\simeq\mathcal{H}_\mathcal{D}\subseteq\mathcal{D}^{\geq0}\subseteq\mathcal{T}^{\geq0}$.
	Moreover, by definition, for every $i=0,\dots,e-1$ we have
	$0=\Ext_A^i(T,X)\simeq\Hom_{\mathcal{D}(A)}(T,X[i])$, hence $X$ belongs in fact to
	$\mathcal{T}^{\geq e}$. Lastly, again by definition, for every $i>e$ we have
	$0=\Ext_A^i(T,X)\simeq\Hom_{\mathcal{D}(A)}(T,X[i])$, hence $X$ belongs to $\mathcal{T}^{\leq
	e}$ as well.

	This proves that, after identifying $A\text{-Mod}\simeq\mathcal{H}_\mathcal{D}$, for every $e=0,\dots,n$
	the $e$-th Miyashita class is
	\[ KE_e = A\text{-Mod}\cap\mathcal{H}_{T}[-e].\]
\end{remark}

Let us now look at Proposition \ref{proposition:ntil:inclusions} in the $n=1$ case. Its proof suggests that we may focus on the
inclusions between the aisles (those between the coaisles being their ``dual'' in the sense of Lemma
\ref{lemma:t-str:aisle-coaisle-duality}).
If $T$ is a 1-tilting module we will then have:
\begin{equation}\tag{$\ast$}\label{eq:inclusions1}
	\mathcal{D}^{\leq -1}\subseteq\mathcal{T}^{\leq0}\subseteq\mathcal{D}^{\leq0}.
\end{equation}
In other words, complexes in $\mathcal{T}^{\leq0}$ are allowed to have any cohomology (with
respect to $\mathcal{D}$, which means the usual complex cohomology $H^i$)
in degrees $\leq-1$ and some kind of cohomology in degree 0, while
they must have 0 cohomology in higher degrees.

\begin{remark}\label{remark:zero-cohomology}
	We may try to characterise $H^0(X^\bullet)$ for $X^\bullet\in\mathcal{T}^{\leq0}$.
Notice that $X^\bullet$ sits in the approximation triangle with respect to $\mathcal{D}$
\[\xymatrix{\delta^{\leq-1}(X^\bullet)\ar[r]&X^\bullet\ar[r]&H^0(X^\bullet)[0]\ar[r]^-{+1}&}\]
where
$H^0(X^\bullet)[0]=H_\mathcal{D}^0(X^\bullet)=\delta^{\geq0}\delta^{\leq0}(X^\bullet)\simeq\delta^{\geq0}(X^\bullet)$
since $X^\bullet\in\mathcal{D}^{\leq0}$. If we apply the homological functor
$\Hom_{\mathcal{D}(A)}(T,?)$ to it, we get the long exact sequence of abelian groups
	\[\resizebox{\textwidth}{!}{\xymatrix@C1pc{\cdots\ar[r]&
	\Hom_{\mathcal{D}(A)}(T,X^\bullet[1])\ar[r]&\Hom_{\mathcal{D}(A)}(T,H^0(X^\bullet)[1])\ar[r]&
	\Hom_{\mathcal{D}(A)}(T,\delta^{\leq-1}(X^\bullet)[2])\ar[r]&\cdots}}\]
The last terms is 0 because $\delta^{\leq-1}(X^\bullet)\in\mathcal{D}^{\leq-1}\subseteq\mathcal{T}^{\leq0}$; similarly, the
first is 0 because $X^\bullet\in\mathcal{T}^{\leq0}$. This means that
\[\Ext_A^1(T,H^0(X^\bullet))\simeq \Hom_{\mathcal{D}(A)}(T,H^0(X^\bullet)[1])=0\]
as well, i.e. that $H^0(X^\bullet)\in KE_0$.
\qed
\end{remark}

The inclusions \eqref{eq:inclusions1} are precisely the hypothesis of the following proposition by
Polishchuk \cite{Polishchuk07}.
\begin{proposition}\label{proposition:ntil:t-str-tilting}
	Let $\mathcal{R},\mathcal{S}$ be two $t$-structures in $\mathcal{D}(A)$ such that
	\[\mathcal{R}^{\leq-1}\subseteq\mathcal{S}^{\leq0}\subseteq\mathcal{R}^{\leq0}\qquad
	\text{(or equivalently }
	  \mathcal{R}^{\geq0}\subseteq\mathcal{S}^{\geq0}\subseteq\mathcal{R}^{\geq-1} \text{)}.\]
	Then $\mathcal{S}$ is obtained by \emph{tilting} $\mathcal{R}$ with respect to a torsion
	pair $(\mathcal{X},\mathcal{Y})$ in the heart $\mathcal{H}_\mathcal{R}$, i.e.
	\begin{align*}
		\mathcal{S}^{\leq0} &= \left\{X^\bullet\in\mathcal{R}^{\leq0}:\; H_\mathcal{R}^0(X^\bullet)\in\mathcal{X}\right\}\\
		\mathcal{S}^{\geq0} &= \left\{X^\bullet\in\mathcal{R}^{\geq-1}:\; H_\mathcal{R}^{-1}(X^\bullet)\in\mathcal{Y}\right\}.
	\end{align*}
	The torsion pair $(\mathcal{X},\mathcal{Y})$ is constructed as
	\[\mathcal{X}=\mathcal{H}_\mathcal{R}\cap\mathcal{S}^{\leq0}=\mathcal{R}^{\geq0}\cap\mathcal{S}^{\leq0}
	\qquad\mathcal{Y}=\mathcal{H}_\mathcal{R}\cap\mathcal{S}^{\geq1}=\mathcal{R}^{\leq0}\cap\mathcal{S}^{\geq1}.\]
\end{proposition}

\begin{remark}
	It can be proved without too much effort that in our case the torsion pair
	$(\mathcal{X},\mathcal{Y})$ in $\mathcal{H}_\mathcal{D}\simeq A\text{-Mod}$ so identified is exactly the pair $(KE_0,KE_1)$
	induced by the 1-tilting module $T$; this confirms Remark
	\ref{remark:zero-cohomology}.
\end{remark}

We would like to use a procedure analogous to the tilting of Proposition
\ref{proposition:ntil:t-str-tilting} in order to link $\mathcal{D}$ and $\mathcal{T}$ in the $n>1$ case.
Notice that if we repeat this tilting operation $n$ times, the first and last of the produced
$t$-structures will be related by the inclusions of Proposition \ref{proposition:ntil:inclusions}. Indeed,
let $\mathcal{R}_0,\dots,\mathcal{R}_n$ be $t$-structures such that $\mathcal{R}_i$ is obtained by
tilting $\mathcal{R}_{i-1}$ with respect to some torsion pair on $\mathcal{H}_{\mathcal{R}_{i-1}}$,
for every $i=1,\dots,n$. Then, we have by construction
\[\mathcal{R}_0^{\leq-1}\subseteq\mathcal{R}_1^{\leq0}\subseteq\mathcal{R}_0^{\leq0}\quad\text{and}\quad
	\mathcal{R}_1^{\leq-1}\subseteq\mathcal{R}_2^{\leq0}\subseteq\mathcal{R}_1^{\leq0}\]
which combined give
\[\mathcal{R}_0^{\leq-2}\subseteq\mathcal{R}_2^{\leq0}\subseteq\mathcal{R}_0^{\leq0}.\]
One can then clearly prove by induction that
\[\mathcal{R}_0^{\leq-n}\subseteq\mathcal{R}_n^{\leq0}\subseteq\mathcal{R}_0^{\leq0}.\]

If $T$ is $n$-tilting, we shall show that the associated $t$-structure $\mathcal{T}$ can
indeed be constructed from $\mathcal{D}$ with this iterated procedure.
To do so, we are going to
construct the ``intermediate'' $t$-structures produced after each tilting.

For $i=0,\dots,n$, consider the strict full subcategories $\mathcal{D}_i^\geq =
\mathcal{D}^{\geq-i}\cap\mathcal{T}^{\geq0}$ (notice that we are working with the coaisles). We have as wanted that
\[\mathcal{D}^{\geq0}=\mathcal{D}_0^\geq\subseteq\mathcal{D}_1^\geq\subseteq\cdots\subseteq\mathcal{D}_n^\geq=\mathcal{T}^{\geq0}\]
and $\mathcal{D}_{i-1}^\geq\subseteq\mathcal{D}_i^\geq\subseteq\mathcal{D}_{i-1}^\geq[1]$ for $i=1,\dots,n$.
The only thing needed to proceed with an iterated application of Proposition
\ref{proposition:ntil:t-str-tilting} is to prove
that these $\mathcal{D}_i^\geq$ are indeed the coaisle of some $t$-structure, for $i=1,\dots,n-1$.

\begin{lemma}\label{lemma:mattiello}
	The $\mathcal{D}_i^\geq = \mathcal{D}^{\geq -i}\cap \mathcal{T}^{\geq 0}$ are coaisles of
	$t$-structures.
\end{lemma}
\begin{proof}
	As we noticed before (see Remark \ref{rmk:d-gen} and the definition of $\mathcal{T}$), we have
	\begin{align*}
		\mathcal{D}^{\geq -i} &=
		\left\{Y^\bullet\in\mathcal{D}(A):\;\Hom_{\mathcal{D}(A)}(A[i],Y^\bullet[j])=0
		\text{ for every } j<0\right\} \\
		\mathcal{T}^{\geq 0} &=
		\left\{Y^\bullet\in\mathcal{D}(A):\;\Hom_{\mathcal{D}(A)}(T,Y^\bullet[j])=0
		\text{ for every } j<0\right\}.
	\end{align*}
	Hence, we have
	\[\mathcal{D}^{\geq -i}\cap\mathcal{T}^{\geq0} =
	\left\{Y^\bullet\in\mathcal{D}(A):\;\Hom_{\mathcal{D}(A)}(T\oplus A[i]),Y^\bullet[j])=0
	\text{ for every } j<0\right\}\]
	which is the coaisle of the $t$-structure generated by $T\oplus A[i]$ in the sense of Proposition
	\ref{prop:t-str:susp}.
\end{proof}

%
%
%

This concludes our previous discussion, making sure that $\mathcal{T}$ can be constructed from
$\mathcal{D}$ with (at most) $n$ applications of the procedure of tilting a $t$-structure with respect
to a torsion pair on its heart.

\section{The t-tree}\label{sec:5}

We are now going to exploit this fact to solve our decomposition problem.

First, we characterise the torsion pairs involved. According to Proposition
\ref{proposition:ntil:t-str-tilting}, at the $i$-th step the $t$-structure $\mathcal{D}_i$ (having coaisle
$\mathcal{D}_i^{\geq0}=\mathcal{D}_i^\geq=\mathcal{D}^{\geq-i}\cap\mathcal{T}^{\geq0}$) is tilted with
respect to the torsion pair
$(\mathcal{X}_i,\mathcal{Y}_i)=(\mathcal{D}_i^{\geq0}\cap\mathcal{D}_{i+1}^{\leq0},
\mathcal{D}_i^{\leq0}\cap\mathcal{D}_{i+1}^{\geq1})$ in the heart $\mathcal{H}_i$ of
$\mathcal{D}_i$, $i=0,\dots,n-1$, thus producing the $t$-structure $\mathcal {D}_{i+1}$.

Now let $X$ be a left $A$-module. As always, we may regard it as a complex concentrated in degree 0, $X[0]$ in
the heart $\mathcal{H}_\mathcal{D}=\mathcal{H}_0$. The first torsion pair
$(\mathcal{X}_0,\mathcal{Y}_0)$ provides then a decomposition
\[ \xymatrix{X_0\ar@{^{(}->}[r]&X\ar@{->>}[r]&X_1}\qquad\text{in }\mathcal{H}_0 \]
with $X_0\in\mathcal{X}_0, X_1\in\mathcal{Y}_0$. Notice that by construction
$\mathcal{X}_0\subseteq\mathcal{H}_1$ and $\mathcal{Y}_0\subseteq\mathcal{H}_1[-1]$ (see Proposition
\ref{proposition:ntil:t-str-tilting});
this means that we can use $(\mathcal{X}_1,\mathcal{Y}_1)$ and
$(\mathcal{X}_1[-1],\mathcal{Y}_1[-1])$ to further
decompose $X_0$ and $X_1$ respectively, obtaining:
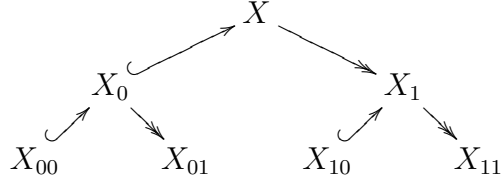
\[
	\xymatrix@R1pc@C.5pc{&&&X\ar@{->>}[rrd]\\
		  &X_0\ar@{^{(}->}[rru]\ar@{->>}[rd]&&&&X_1\ar@{->>}[rd]\\
		  X_{00}\ar@{^{(}->}[ru]&&X_{01}&&X_{10}\ar@{^{(}->}[ru]&&X_{11} }
\]
with the exact sequences in the respective abelian categories:
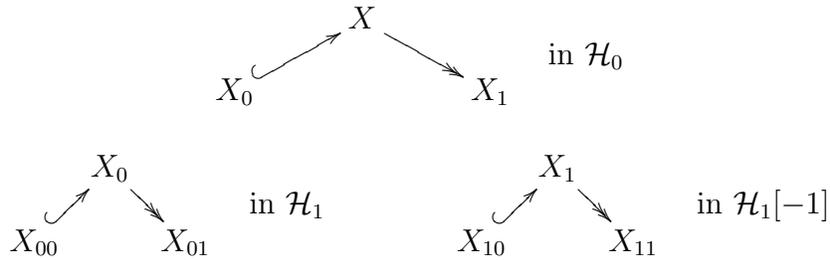
\[ \vcenter{\xymatrix@R1pc@C.5pc{&&X\ar@{->>}[drr]\\ X_0\ar@{^{(}->}[urr]&\quad&&\quad&X_1
	}}\quad \text{in }\mathcal{H}_0\]
\[ \vcenter{\xymatrix@R1pc@C.5pc{&X_0\ar@{->>}[dr]\\ X_{00}\ar@{^{(}->}[ur]&&X_{01}\\
	}}\quad \text{in }\mathcal{H}_1\qquad\qquad
   \vcenter{\xymatrix@R1pc@C.5pc{&X_1\ar@{->>}[dr]\\ X_{10}\ar@{^{(}->}[ur]&&X_{11}
	}}\quad \text{in }\mathcal{H}_1[-1]\]
Now, notice again that since $\mathcal{X}_1\subseteq\mathcal{H}_2$ and
$\mathcal{Y}_1\subseteq\mathcal{H}_2[-1]$, we have that $X_{00}\in\mathcal{H}_2$,
$X_{01},X_{10}\in\mathcal{H}_2[-(0+1)]=\mathcal{H}_2[-(1+0)]$ and $X_{11}\in\mathcal{H}_2[-(1+1)]$.

By induction, by decomposing each $X_{b_1\dots b_i}$ with respect to the torsion pair
$(\mathcal{X}_i[-(b_1+\dots+b_i)],\mathcal{Y}_i[-(b_1+\dots +b_i)])$ in  $\mathcal{H}_i[-(b_1+\dots+b_i)]$ we obtain objects
$X_{b_1\dots b_i0}\in\mathcal{H}_{i+1}[-(b_1+\dots+b_i)]$ and $X_{b_1\dots
b_i1}\in\mathcal{H}_{i+1}[-(b_1+\dots+b_i+1)]$.

After $n$ steps, we obtain the complete diagram
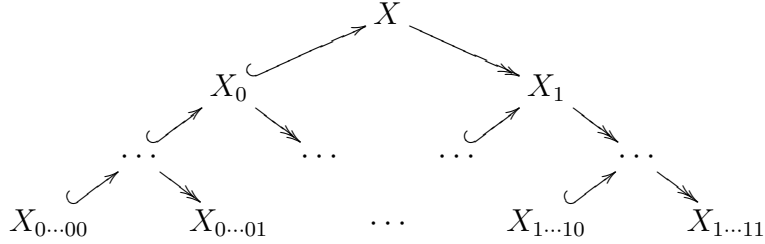
\[
\xymatrix@R1pc@C.5pc{
	&&&&X\ar@{->>}[drr]\\
	&&X_0\ar@{^{(}->}[urr]\ar@{->>}[dr]&&&&X_1\ar@{->>}[dr]\\
	&\cdots\ar@{^{(}->}[ur]\ar@{->>}[dr]&&\cdots&&\cdots\ar@{^{(}->}[ur]&&\cdots\ar@{->>}[dr]\\
	X_{0\cdots00}\ar@{^{(}->}[ur]&&X_{0\cdots01}&&\cdots&&X_{1\cdots10}\ar@{^{(}->}[ur]&&X_{1\cdots11}
	}
\]
with $n+1$ rows, which is called the \emph{t-tree} of $X$ with respect to the $t$-structure induced by the tilting
module $T$.

We claim that this construction in some sense solves our decomposition problem. Indeed, by
construction each object $X_{b_1\cdots b_n}$ in the last row (called a \emph{t-leaf}) belongs to
$\mathcal{H}_{n}[-(b_1+\cdots+b_n)]=\mathcal{H}_\mathcal{T}[-(b_1+\cdots+b_n)]$: as noted in Remark
\ref{remark:MiyashitaHearts} these shifted hearts are extensions of the Miyashita classes: $KE_{b_1+\cdots+b_n}=A\text{-Mod}\cap \mathcal{H}_\mathcal{T}[-(b_1+\cdots+b_n)]$.
Moreover, these shifted hearts are obtained by adding only non-module objects (i.e., objects of $\mathcal{D}(A)$
outside of $\mathcal{H}_\mathcal{D}$) to the corresponding Miyashita class; for this reason they are less artificial than other
enlargments, and instead shed a new light on the Miyashita classes. The latter can indeed be
regarded as the piece of the shifted hearts of $\mathcal{T}$ visible in the category of modules.

\begin{example}\label{ex:ttree}
	We recall one last time the situation considered in Example \ref{ex:2tilt} to show an application of the
	construction of the t-tree; we will do it for the simple module
	${\begin{smallmatrix}2\end{smallmatrix}}$ again.

	First, a computation shows that the indecomposable complexes in $\mathcal{D}(A)$ are
	(shifts of):
	\[ \left\{{\begin{smallmatrix}1\end{smallmatrix}}, {\begin{smallmatrix}2\end{smallmatrix}},
		{\begin{smallmatrix}3\end{smallmatrix}}, {\begin{smallmatrix}1\\2\end{smallmatrix}},
		{\begin{smallmatrix}2\\3\end{smallmatrix}},
		{\begin{smallmatrix}2\\3\end{smallmatrix}}\to{\begin{smallmatrix}1\\2\end{smallmatrix}}
			\right\}. \]
	Since we know that $\mathcal{D}^{\geq 0}\subseteq\mathcal{T}^{\geq 0}$, any bounded below
	complex will belong to $\mathcal{T}^{\geq0}$, up to shifiting it enough to the right. We can
	then check for each of the indecomposable complexes what is their leftmost shift which still
	belongs to $\mathcal{T}^{\geq0}$; with an easy computation, the following is the result:
	\[\mathcal{T}^{\geq 0} =\langle
		{\begin{smallmatrix}1\end{smallmatrix}}, {\begin{smallmatrix}2\end{smallmatrix}},
		{\begin{smallmatrix}3\end{smallmatrix}}[2], {\begin{smallmatrix}1\\2\end{smallmatrix}},
		{\begin{smallmatrix}2\\3\end{smallmatrix}},
		{\begin{smallmatrix}2\\3\end{smallmatrix}}\to
			\overset{\bullet}{\begin{smallmatrix}1\\2\end{smallmatrix}}
		\rangle \]
	where the dot over a complex indicates its degree 0. The angle brackets will be used to
	denote the closure under direct sums and negative shifts.

	Following the construction, we can compute the intermediate coaisles:
	\begin{align*}
	\mathcal{D}_0^{\geq0} &= \langle
		{\begin{smallmatrix}1\end{smallmatrix}}, {\begin{smallmatrix}2\end{smallmatrix}},
		{\begin{smallmatrix}3\end{smallmatrix}}, {\begin{smallmatrix}1\\2\end{smallmatrix}},
		{\begin{smallmatrix}2\\3\end{smallmatrix}},
		\overset{\bullet}{\begin{smallmatrix}2\\3\end{smallmatrix}}\to
			{\begin{smallmatrix}1\\2\end{smallmatrix}} \rangle = \mathcal{D}^{\geq0} \\
	\mathcal{D}_1^{\geq0} &= \langle
		{\begin{smallmatrix}1\end{smallmatrix}}, {\begin{smallmatrix}2\end{smallmatrix}},
		{\begin{smallmatrix}3[1]\end{smallmatrix}}, {\begin{smallmatrix}1\\2\end{smallmatrix}},
		{\begin{smallmatrix}2\\3\end{smallmatrix}},
		{\begin{smallmatrix}2\\3\end{smallmatrix}} \to
			\overset{\bullet}{\begin{smallmatrix}1\\2\end{smallmatrix}} \rangle\\
	\mathcal{D}_2^{\geq0} &= \langle
		{\begin{smallmatrix}1\end{smallmatrix}}, {\begin{smallmatrix}2\end{smallmatrix}},
		{\begin{smallmatrix}3\end{smallmatrix}}[2], {\begin{smallmatrix}1\\2\end{smallmatrix}},
		{\begin{smallmatrix}2\\3\end{smallmatrix}},
		{\begin{smallmatrix}2\\3\end{smallmatrix}}\to
			\overset{\bullet}{\begin{smallmatrix}1\\2\end{smallmatrix}} \rangle =
				\mathcal{T}^{\geq0}.
	\end{align*}
	Now we compute the hearts of the respective $t$-structures: to do this, we use Lemma
	\ref{lemma:t-str:aisle-coaisle-duality}. An object $X$ of $\mathcal{D}_i^{\geq0}$ will belong to
	$\mathcal{H}_i$ if and only if $\Hom_{\mathcal{D}(A)}(X,Y)=0$ for every
	$Y\in\mathcal{D}_i^{\geq 1}=\mathcal{D}_i^{\geq0}[-1]$. In particular, it is easy to see
	that we must look for objects of the heart only among the ``leftmost shifts'' we have
	listed. The resulting computation gives (only indecomposable objects are listed):
	\begin{align*}
	\mathcal{H}_0 &= \left\{
		{\begin{smallmatrix}1\end{smallmatrix}}, {\begin{smallmatrix}2\end{smallmatrix}},
		{\begin{smallmatrix}3\end{smallmatrix}}, {\begin{smallmatrix}1\\2\end{smallmatrix}},
		{\begin{smallmatrix}2\\3\end{smallmatrix}} \right\} = \mathcal{H}_\mathcal{D} = A\text{-Mod}\\
	\mathcal{H}_1 &= \left\{
		{\begin{smallmatrix}1\end{smallmatrix}}, {\begin{smallmatrix}2\end{smallmatrix}},
		{\begin{smallmatrix}3\end{smallmatrix}}[1], {\begin{smallmatrix}1\\2\end{smallmatrix}},
		{\begin{smallmatrix}2\\3\end{smallmatrix}},
		{\begin{smallmatrix}2\\3\end{smallmatrix}} \to
			\overset{\bullet}{\begin{smallmatrix}1\\2\end{smallmatrix}} \right\}\\
	\mathcal{H}_2 &= \left\{
		{\begin{smallmatrix}1\end{smallmatrix}}, {\begin{smallmatrix}3\end{smallmatrix}}[2],
		{\begin{smallmatrix}1\\2\end{smallmatrix}}, {\begin{smallmatrix}2\\3\end{smallmatrix}},
		{\begin{smallmatrix}2\\3\end{smallmatrix}} \to
			\overset{\bullet}{\begin{smallmatrix}1\\2\end{smallmatrix}}\right\}=\mathcal{H}_\mathcal{T}.
	\end{align*}
	Notice that neither ${\begin{smallmatrix}2\end{smallmatrix}}$ nor its shifts belong to
		$\mathcal{H}_\mathcal{T}$,
	which means exactly that it does not belong to any Miyashita class.

	Lastly, we can compute the torsion pairs $(\mathcal{X}_i,\mathcal{Y}_i)$ in $\mathcal{H}_i$,
	for $i=0,1$. We have:
	\begin{align*}
	\mathcal{X}_0 &= \mathcal{H}_0\cap\mathcal{H}_1 = \left\{
		{\begin{smallmatrix}1\end{smallmatrix}}, {\begin{smallmatrix}2\end{smallmatrix}},
		{\begin{smallmatrix}1\\2\end{smallmatrix}}, {\begin{smallmatrix}2\\3\end{smallmatrix}}
			\right\}, \qquad
	\mathcal{Y}_0 = \mathcal{H}_0\cap\mathcal{H}_1[-1] = \left\{
		{\begin{smallmatrix}3\end{smallmatrix}}\right\} \\
	\mathcal{X}_1 &= \mathcal{H}_1\cap\mathcal{H}_2 = \left\{
		{\begin{smallmatrix}1\end{smallmatrix}}, {\begin{smallmatrix}1\\2\end{smallmatrix}},
		{\begin{smallmatrix}2\\3\end{smallmatrix}},
		{\begin{smallmatrix}2\\3\end{smallmatrix}}\to
			\overset{\bullet}{\begin{smallmatrix}1\\2\end{smallmatrix}} \right\}, \qquad
	\mathcal{Y}_1 = \mathcal{H}_1\cap\mathcal{H}_2[-1] = \left\{
		{\begin{smallmatrix}3\end{smallmatrix}}[1] \right\}.
	\end{align*}
	The t-tree for the module ${\begin{smallmatrix}2\end{smallmatrix}}$ is then:
	
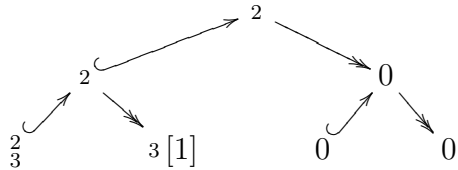
\[ \xymatrix@C1pc@R1pc{
		&&&{\begin{smallmatrix}2\end{smallmatrix}}\ar@{->>}[drr]\\
		&{\begin{smallmatrix}2\end{smallmatrix}}\ar@{^{(}->}[urr]\ar@{->>}[dr]
			&&&&0\ar@{->>}[dr]\\
		{\begin{smallmatrix}2\\3\end{smallmatrix}}\ar@{^{(}->}[ur]
			&&{\begin{smallmatrix}3\end{smallmatrix}}[1]
			&&0\ar@{^{(}->}[ur]&&0
		}\]
	where the bottom left exact sequence is that associated to the distinguished triangle
	\[\xymatrix{{\begin{smallmatrix}2\\3\end{smallmatrix}}\ar[r]&
				{\begin{smallmatrix}2\end{smallmatrix}}\ar[r]&
				{\begin{smallmatrix}3\end{smallmatrix}}[1]\ar[r]^-{+1}&
		}.\]
	Notice that this triangle can be shifted to become
	$\xymatrix@1{{\begin{smallmatrix}3\end{smallmatrix}}\ar[r]&
		{\begin{smallmatrix}2\\3\end{smallmatrix}}\ar[r]&
			{\begin{smallmatrix}2\end{smallmatrix}}\ar[r]^-{+1}&
		}$,
	which can be read as a short exact sequence of modules. This says that
	${\begin{smallmatrix}2\end{smallmatrix}}$ is realised as the cokernel of the monomorphism
	${\begin{smallmatrix}3\end{smallmatrix}}\to{\begin{smallmatrix}2\\3\end{smallmatrix}}$,
		which is what was found following Jensen, Madsen and Su in Example \ref{ex:JMS}.
\end{example}

\newpage
\hrule
\vskip.2truecm

F. Mattiello - Dipartimento di Matematica ``Tullio Levi-Civita'', Universit\`a degli studi di Padova, fran.mattiello@gmail.com
	      
S. Pavon - Dipartimento di Matematica ``Tullio Levi-Civita'', Universit\`a degli studi di Padova, sergio.pavon@math.unipd.it
                                         
A. Tonolo - Dipartimento di Scienze Statistiche, Universit\`a degli studi di Padova, alberto.tonolo@unipd.it
\vskip.2truecm
\hrule
\end{document}